\documentclass[12pt]{article}

\usepackage{times}

\usepackage{dsfont}
\usepackage[shortlabels]{enumitem}
\usepackage{enumitem}
\usepackage[utf8]{inputenc}
\usepackage{commath}
\usepackage{amsthm, amscd}
\usepackage{amsmath}
\usepackage{amssymb}
\usepackage{cite}
\usepackage{setspace}
\usepackage{ stmaryrd }
\usepackage{tikz-cd}

\usepackage{tocloft}
\usepackage{sectsty}
\usepackage{xparse}

\newcommand*\tasklabelformat[1]{#1)}

\usepackage{titlesec}

\usepackage{tasks}[newest]
\settasks{
  label = \alph* ,
  label-format = \tasklabelformat ,
  label-width  = 12pt
}

\usepackage{bigints}
\usepackage{comment}
\usepackage{mathtools}
\usepackage{mathrsfs}
\usepackage{fancyhdr}

\allowdisplaybreaks[4]

\numberwithin{equation}{section}
\usepackage{etoolbox}
\patchcmd{\thebibliography}
  {\settowidth}
  {\setlength{\itemsep}{0pt plus -10pt}\settowidth}
  {}{}
\apptocmd{\thebibliography}
  {
  }
  {}{}
\makeatletter
\newtheorem*{rep@theorem}{\rep@title}
\newcommand{\newreptheorem}[2]{%
\newenvironment{rep#1}[1]{%
 \def\rep@title{#2 \ref{##1}}%
 \begin{rep@theorem}}%
 {\end{rep@theorem}}}
\makeatother

\theoremstyle{theorem}

\newreptheorem{theorem}{Theorem}
\newtheorem{thm}{Theorem}[section]
\newtheorem*{thm*}{Theorem}
\theoremstyle{definition}
\newtheorem{prop}[thm]{Proposition}
\newtheorem*{prop*}{Proposition}

\newtheorem{lem}[thm]{Lemma}

\newtheorem*{cor*}{Corollary}
\theoremstyle{remark}

\newtheorem{rem}[thm]{Remark}

\title{\vspace*{-1.5cm} Metric entropy of the space of holomorphic functions}

\author
{Siarhei Finski
}

\date{}

\usepackage[%
    left=1in,%
    right=1in,%
    top=1.1in,%
    bottom=0.8in,%
    paperheight=11in,%
    paperwidth=8.5in%
]{geometry}

\newcommand{\imun} {\sqrt{-1}}

\newcommand{\comp}{\mathbb{C}}
\newcommand{\real}{\mathbb{R}}

\newcommand{\nat}{\mathbb{N}}

\newcommand{\dbar}{ \overline{\partial} }

\newcommand{\ddc}{\mathrm{d} \mathrm{d}^c}

\newcommand{\psh}{\operatorname{PSH}}

\makeatletter
\DeclareFontFamily{OMX}{MnSymbolE}{}
\DeclareSymbolFont{MnLargeSymbols}{OMX}{MnSymbolE}{m}{n}
\SetSymbolFont{MnLargeSymbols}{bold}{OMX}{MnSymbolE}{b}{n}
\DeclareFontShape{OMX}{MnSymbolE}{m}{n}{
    <-6>  MnSymbolE5
   <6-7>  MnSymbolE6
   <7-8>  MnSymbolE7
   <8-9>  MnSymbolE8
   <9-10> MnSymbolE9
  <10-12> MnSymbolE10
  <12->   MnSymbolE12
}{}
\DeclareFontShape{OMX}{MnSymbolE}{b}{n}{
    <-6>  MnSymbolE-Bold5
   <6-7>  MnSymbolE-Bold6
   <7-8>  MnSymbolE-Bold7
   <8-9>  MnSymbolE-Bold8
   <9-10> MnSymbolE-Bold9
  <10-12> MnSymbolE-Bold10
  <12->   MnSymbolE-Bold12
}{}

\let\llangle\@undefined
\let\rrangle\@undefined
\DeclareMathDelimiter{\llangle}{\mathopen}%
                     {MnLargeSymbols}{'164}{MnLargeSymbols}{'164}
\DeclareMathDelimiter{\rrangle}{\mathclose}%
                     {MnLargeSymbols}{'171}{MnLargeSymbols}{'171}
\makeatother

\newenvironment{sciabstract}{}

\cftsetindents{section}{0em}{1.7em}

\sectionfont{\large}

\titlespacing*{\section}
{0pt}{5pt}{5pt}

\begin{document}

\maketitle 

\vspace*{-0.7cm}

\vspace*{0.3cm}

\begin{sciabstract}
  \textbf{Abstract.}
  By Montel's theorem, the continuous functions on a compact subset of a complex manifold that admit uniformly bounded holomorphic extensions to the manifold form a compact set in the uniform topology. 
  We render this compactness quantitative by determining the asymptotics of the associated metric entropy, 
thereby giving a new solution to a problem of Kolmogorov for domains in $\comp^n$ and extending the solution to domains in arbitrary Stein manifolds.
\end{sciabstract}

\pagestyle{fancy}
\lhead{}
\chead{Metric entropy of the space of holomorphic functions}
\rhead{\thepage}
\cfoot{}


\newcommand{\Addresses}{{
  \bigskip
  \footnotesize
  \noindent \textsc{Siarhei Finski, Centre de Mathématiques Laurent Schwartz (CMLS), CNRS, École polytechnique, Institut Polytechnique de Paris, Palaiseau, France}\par\nopagebreak
  \noindent  \textit{E-mail }: \texttt{siarhei.finski@polytechnique.edu} / \texttt{finski.siarhei@gmail.com}.
}}

\vspace*{0.25cm}

\par\noindent\rule{1.25em}{0.4pt} \textbf{Table of contents} \hrulefill

\vspace*{-1.5cm}

\tableofcontents

\vspace*{-0.2cm}

\noindent \hrulefill


\section{Introduction}\label{sect_intro}
	This paper concerns the size of certain function spaces, measured in the sense of \textit{metric entropy}, also called \textit{$\varepsilon$-entropy}. 
	For a compact subset $A$ of a metric space $(M, d)$ and $\varepsilon > 0$, it is defined by
	\begin{equation}
		H_{\varepsilon}(A, M)
		:=
		\log N_{\varepsilon}(A, M),
	\end{equation}
	where $N_{\varepsilon}(A, M)$ is the least number of balls of radius $\varepsilon$ needed to cover $A$.
	\par 
	This definition is natural from an information-theoretic standpoint: it quantifies the amount of information that must be transmitted to specify an element of $A$ up to precision $\varepsilon$.
	The notion goes back to Kolmogorov \cite{KolmEntr, KolmTikh}, building on earlier work of Pontrjagin-Schnirelmann \cite{PontrSchnir}.
	\par 
	For function spaces arising from real analysis, the study of $\varepsilon$-entropy was largely motivated by developments surrounding Hilbert's thirteenth problem. 
	Most notably, Vitushkin \cite{VitushkinThese} established the existence of smooth functions of $n \geq 2$ variables that cannot be represented as superpositions of smooth functions of fewer variables. 
	Kolmogorov \cite{KolmogorovVitush} subsequently simplified part of the proof by estimating the $\varepsilon$-entropy of the corresponding function spaces; see \cite{VitushkinBook, VitushkinHalfCent} for a detailed account.
	\par 
	For function spaces arising in complex analysis, the estimation of the corresponding $\varepsilon$-entropies is called \textit{a problem of Kolmogorov}, and it has attracted significant interest; see \cite{KolmEntr, Babenko, Erokhin, LevinTikh, ZahariutaBases, Nguyen, WidomL2Kolm, Widom, ZakharyutaSkiba, NivocheConj, BandNivoc}.
	To explain the problem, we fix a compact subset $K$ of a \textit{domain} $\Omega$ -- that is, a connected open subset -- of a Stein manifold.
	Let $H^0(\Omega)$ be the space of holomorphic functions on $\Omega$.
	Inside the space $\mathscr{C}(K)$ of continuous functions on $K$, we define the following space
	\begin{equation}\label{eq_ak_omeg}
		A_K^{\Omega} := 
		\Big\{ f \in \mathscr{C}(K) : \text{there is } \widetilde{f} \in H^0(\Omega) \text{ such that } \widetilde{f}|_K = f, |\widetilde{f}(x)| \leq 1 \text{ for any } x \in \Omega \Big\}.
	\end{equation}
	\par 
	By Montel's theorem, $A_K^{\Omega}$ is a compact subset of $\mathscr{C}(K)$, the latter viewed as a metric space under the sup-norm.
	In particular, the $\varepsilon$-entropy of the pair $(A_K^{\Omega}, \mathscr{C}(K))$ is well defined. 
	A problem of Kolmogorov concerns the asymptotic study of this $\varepsilon$-entropy, as $\varepsilon \to 0$.
	\par
	This asymptotics is closely related to the notion of an \textit{envelope of holomorphy}, $\widehat{\Omega}$, of $\Omega$.
	Recall that the latter means that $\widehat{\Omega}$ is a Stein manifold such that $\Omega$ is biholomorphic to an open subset of $\widehat{\Omega}$, and for every $f \in H^0(\Omega)$, there is a (unique) $\widehat{f} \in H^0(\widehat{\Omega})$ such that $\widehat{f}|_{\Omega} = f$.
	It was established by Cartan-Thullen, cf. \cite[Theorems 5.4.3, 5.4.5]{Hormander}, (for domains in $\comp^n$) and Rossi \cite{RossiEnv} (for domains in general Stein manifolds) that there is a unique (up to a biholomorphism) such $\widehat{\Omega}$.
	\par 
	Another central object in the problem of Kolmogorov is the so-called relative capacity, introduced in higher dimensions by Bedford-Taylor \cite{BedfordTaylor}. 
	To recall its definition, we say that $\Omega$ is \textit{hyperconvex} if it carries a continuous plurisubharmonic (\textit{psh}) exhaustion function $\rho \colon \Omega \to ]-\infty, 0[$, meaning that the sublevel sets $\{\rho < c\}$, $c < 0$, are relatively compact in $\Omega$. 
	For a compact subset $K$ of such a domain $\Omega$, the \emph{relative capacity} is then defined as
	\begin{equation}\label{eq_rel_cap}
		C(K,\Omega)
		=
	  	\sup \Big\{
	    \int_{K} (dd^c u)^n
	    :
	    u \in \mathrm{PSH}(\Omega),\ -1 \leq u \leq 0
	  	\Big\},
	\end{equation}
	where $\mathrm{PSH}(\Omega)$ is the cone of psh functions on $\Omega$; see (\ref{eq_defn_dc}) for the definition of $d^c$.
	\par 
	Recall that any Stein manifold $Y$ admits an exhaustion by relatively compact hyperconvex domains $Y_i$, $i \in \nat$; for instance, each $Y_i$ may be taken to be a sublevel set of a strictly psh exhaustion function, whose existence is furnished by Grauert's solution of the Levi problem, cf. \cite[Theorem 5.2.10]{Hormander}.
	We then extend the definition of the relative capacity as
	\begin{equation}\label{eq_rel_cap_gen_dd}
		C(K, Y)
		:=
		\lim_{i \to \infty} C(K, Y_i).
	\end{equation}
	The monotonicity of the relative capacity under inclusion of hyperconvex domains shows that the $C(K, Y_i)$ decrease, so the limit above exists.
	Similar reasoning shows that it is independent of the choice of $Y_i$. 
	Also, by Lemma \ref{lem_cap_reg}, this definition is consistent with (\ref{eq_rel_cap}) when $Y$ is hyperconvex.
	\par 
	Finally, we introduce the quantity
	\begin{equation}\label{eq_und_c}
		\underline{C}(K, \Omega)
		:=
		\sup C(K, \Omega'),
	\end{equation}
	where the supremum is taken over all hyperconvex Stein manifolds $\Omega'$ containing $\Omega$ as an open relatively compact subset.
	If no such $\Omega'$ exists, we set $\underline{C}(K, \Omega) := 0$; note that such manifolds $\Omega'$ always exist when $\Omega$ is a relatively compact subset of a Stein manifold.
	We are now in a position to state the main result of this article, giving an answer to a problem of Kolmogorov.
	\begin{thm}\label{thm_entr}
		Let $\Omega$ be a domain in a Stein manifold $X$ of dimension $n$, and let $K \subset \Omega$ be a non-pluripolar compact subset. Denoting by $\widehat{\Omega}$ the envelope of holomorphy of $\Omega$, we have
		\begin{equation}\label{eq_thm_entr}
			\frac{2 C(K, \widehat{\Omega})}{(n + 1)!}
			\geq
			\limsup_{\varepsilon \to 0}
			\frac{H_{\varepsilon}(A_K^{\Omega}, \mathscr{C}(K))}{(\log(\varepsilon^{-1}))^{n + 1}}
			\geq
			\liminf_{\varepsilon \to 0}
			\frac{H_{\varepsilon}(A_K^{\Omega}, \mathscr{C}(K))}{(\log(\varepsilon^{-1}))^{n + 1}}
			\geq
			\frac{2 \underline{C}(K, \Omega)}{(n + 1)!}.
		\end{equation}
	\end{thm}
	\begin{rem}
		a) Recall that a domain $\Omega$ in a Stein manifold $Y$ is called \textit{strictly hyperconvex} if there exist an open neighborhood $\Omega'$ of the closure $\overline{\Omega}$ of $\Omega$ in $Y$ and a continuous psh exhaustion function $\rho: \Omega' \to ]- \infty, 1[$ such that for any $r < 1$, for the sublevel sets
		\begin{equation}\label{eq_sublevel}
			\Omega_r := \{x \in \Omega' : \rho(x) < r \},
		\end{equation}
		we have $\Omega = \Omega_0$.
		As will be discussed in Remark \ref{rem_c_above_cont}, when $\widehat{\Omega}$ is strictly hyperconvex, we have
		\begin{equation}\label{eq_cap_ext_int_eq}
			C(K, \widehat{\Omega})
			=
			\underline{C}(K, \Omega),
		\end{equation}
		and so Theorem \ref{thm_entr} determines the asymptotics of $H_{\varepsilon}(A_K^{\Omega}, \mathscr{C}(K))$ for such $\Omega$.
		It would be interesting to find an intrinsic characterization of the domains $\Omega$ satisfying (\ref{eq_cap_ext_int_eq}). 
		A partial result in this direction is the following. 
		Since the envelope of holomorphy of a Stein domain $\Omega$ coincides with $\Omega$ itself, see \cite[Theorem 5.4.2]{Hormander}, Theorem \ref{thm_entr} together with (\ref{eq_cap_ext_int_eq}) yields the following asymptotics of the metric entropy for any strictly hyperconvex domain $\Omega$ in $X$:
		\begin{equation}\label{thm_hyperconvex_vthm}
			\lim_{\varepsilon \to 0}
			\frac{H_{\varepsilon}(A_K^{\Omega}, \mathscr{C}(K))}{(\log(\varepsilon^{-1}))^{n + 1}}
			=
			\frac{2 C(K, \Omega)}{(n + 1)!}.
		\end{equation}
		\par 
		b) We must rely in (\ref{eq_thm_entr}) on the capacity from (\ref{eq_rel_cap_gen_dd}), since an envelope of holomorphy need not be hyperconvex. 
		Indeed, the punctured unit disc $\Omega \subset \comp$ satisfies $\Omega = \widehat{\Omega}$, yet it is not hyperconvex: every subharmonic function on $\Omega$ that is bounded from above extends subharmonically  across the puncture, so $\Omega$ admits no negative subharmonic exhaustion function.
	\end{rem}
	\par 
	Let us place Theorem \ref{thm_entr} and (\ref{thm_hyperconvex_vthm}) in the context of some previous works.
	Following a conjecture of Kolmogorov (which was motivated by some previous explicit calculations of Vitushkin, \cite{VitushkinEntr}, cf. \cite[p. 134]{KolmSelWorksIII}), a version of Theorem \ref{thm_entr} was established by Babenko \cite{Babenko} and Erokhin \cite{Erokhin} for $X = \comp$.
	Later Zakharyuta-Skiba \cite{ZakharyutaSkiba} extended it for arbitrary Riemann surfaces $X$.
	For $n \geq 2$, the identity (\ref{thm_hyperconvex_vthm}) was established by Zakharyuta \cite{ZakharyutaPotential}, conditional on a certain pluripotential-theoretic conjecture, which was later established by Nivoche \cite{NivocheConj} for $X = \comp^n$.
	Bandtlow-Nivoche \cite{BandNivoc} then gave a different proof of (\ref{thm_hyperconvex_vthm}), still in the setting $X = \comp^n$.
	\par 
	Hence, (\ref{thm_hyperconvex_vthm}) both extends \cite{ZakharyutaPotential, NivocheConj, BandNivoc} to domains in arbitrary Stein manifolds and gives a new, independent proof for domains in $\comp^n$.
	\par 
	Note that already for $n = 1$ there exist domains $\Omega$ that embed into an open Riemann surface (which is Stein) but not into $\comp$ (one can show that the intersection pairing on the first cohomology group $H^1(\Omega)$ of $\Omega$ vanishes for domains $\Omega$ in $\comp$).
	The previous approaches to (\ref{thm_hyperconvex_vthm}) do not appear to generalize readily to domains in arbitrary Stein manifolds. 
	The argument of \cite{BandNivoc}, for example, uses the explicit Bergman-Weil formula in $\comp^n$. 
	This formula can be generalized to arbitrary Stein manifolds, cf. \cite[\S 4]{HenkinLeiterer}, but in a substantially different form: the difference of points is not well defined on a general manifold, so the analysis of \cite[Lemma 3.5]{BandNivoc} would have to be reworked.
	\par 
	Furthermore, while we deduce Theorem \ref{thm_entr} from (\ref{thm_hyperconvex_vthm}), the proof of Theorem \ref{thm_entr} requires the extension of (\ref{thm_hyperconvex_vthm}) to arbitrary Stein manifolds $X$ even when $\Omega$ is a domain in $\comp^n$.
	This is because (\ref{thm_hyperconvex_vthm}) is applied to hyperconvex exhaustions of $\widehat{\Omega}$, and by Stout \cite{StoutDomain} there exist domains $\Omega$ in $\comp^n$ for which $\widehat{\Omega}$ is not biholomorphic to any domain in $\comp^n$.
	\par 
	Moreover, even when $\Omega$ is embedded in $\comp^n$, our argument for (\ref{thm_hyperconvex_vthm}) proceeds by reduction to domains in affine manifolds, which need not embed in $\comp^n$.
	\par 
	Another novelty of the current work is that, through Theorem \ref{thm_rel_cap}, it gives a conceptual explanation for the appearance of relative capacity in the $\varepsilon$-entropy estimates, whereas previously this appearance rested on explicit computation: \cite[\S 5, \S 8]{ZakharyutaSurvey} reduced the general case to domains associated with pluricomplex Green functions, and \cite[\S 3.1]{BandNivoc} reduced it to special analytic polyhedra, the relative capacity being computed explicitly in each case.
	\par 
	To explain our approach to (\ref{thm_hyperconvex_vthm}), recall that a domain $\Omega$ in $X$ is called a \textit{Runge domain} if it is Stein and $H^0(X)$ is dense in $H^0(\Omega)$ (with respect to the uniform topology on compact subsets of $\Omega$).
	A strictly hyperconvex domain $\Omega$ in $X$ is \textit{strictly Runge} if in the notation introduced before (\ref{eq_sublevel}), the domain $\Omega'$ is Runge in $X$ (every strictly Runge domain is Runge; see Remark \ref{eq_runge_dom_emb}).
	\par
	The main idea of our proof of (\ref{thm_hyperconvex_vthm}) is to reduce the estimation of the $\varepsilon$-entropy of $A_K^{\Omega}$ to its finite-dimensional algebraic approximations.
	This reduction proceeds in several steps. 
	First, relying on some well-known results on the geometry of strictly hyperconvex domains together with a theorem of Stout \cite{StoutStein}, cf. Demailly-Lempert-Shiffman \cite{DemLempShiff}, we show that it suffices to treat the case where $X$ is an affine manifold in $\comp^m$ for some $m \in \nat$, and $\Omega$ is a relatively compact strictly Runge domain in $X$.
	Roughly, it means that it suffices to estimate the $\varepsilon$-entropy of a space analogous to $A_K^{\Omega}$ but with $H^0(X)$ in its definition instead of $H^0(\Omega)$.
	\par 
	The second step consists in replacing the space $H^0(X)$ by its finite-dimensional approximations $A_k[X]$, $k \in \nat$, where $A_k[X]$ is the restriction to $X$ of the space of polynomials of total degree $\leq k$ on $\comp^m$.
	For this, we rely on Siciak's \cite{SiciakEnvFirst} result on optimal polynomial approximations of holomorphic functions, extended to the setting of affine varieties by Zeriahi \cite{ZeriahiSic}.
	\par 
	The calculation of the $\varepsilon$-entropy from (\ref{thm_hyperconvex_vthm}) then reduces to the calculation of the $\varepsilon$-entropy on the spaces $A_k[X]$ associated with two $\sup$-norms: one on $K$ and the other on $\Omega$.
	We show that this eventually boils down to the comparison of the volumes of the unit balls of certain $\sup$-norms built from the two above, which can be carried out using Berman-Boucksom \cite{BermanBouckBalls}.
	\par 
	The application of the latter result, however, does not feature in a natural way the relative capacity $C(K, \Omega)$. 
	The last step of our proof is, hence, to find an interpretation of the latter in terms of the so-called relative Monge-Ampère energy which appears in the comparison of the volumes of the unit balls.
	This step might be of independent interest, so we explain it below.
	\par 
	For this, recall that for an affine manifold $X \subset \comp^m$, the \emph{Lelong class} of psh functions of logarithmic growth on $X$ is defined as	
	\begin{equation}
		 \mathscr{L}(X) := \Big \{ u \in \psh(X) : \sup_{x \in X} \big(u(x) - \log(\| x \| + 1) \big) < +\infty \Big \},
	\end{equation}
	where $\|\cdot\|$ is the Euclidean norm on $\comp^m$ restricted to $X$.
	\par 
	The \emph{Siciak-Zahariuta extremal function} of the closure $\overline{\Omega} \subset X$ of $\Omega$ is defined for $z \in X$ as
	\begin{equation}\label{eq_extr}
		V_{\Omega}(z) := \sup\big\{ u(z) : u \in \mathscr{L}(X),\ u|_{\overline{\Omega}} \leq 0 \,\big\}.
	\end{equation}
	For any $c > 0$, $z \in X$, we also define the following relative extremal function: for $z \in X$, we let
	\begin{equation}\label{eq_rel_extr}
		V_{c K, \Omega}(z) := \sup\big\{ u(z) : u \in \mathscr{L}(X),\ u|_{K} \leq - c, u|_{\overline{\Omega}} \leq 0 \,\big\}.
	\end{equation}
	Define the relative Monge-Ampère energy as
	\begin{equation}\label{eq_energy}
		E(V_{\Omega}) - E(V_{c K, \Omega})
		=
		\frac{2}{(n + 1)!}
		\sum_{i = 0}^{n} \int_X \big( V_{\Omega}^* - V_{c K, \Omega}^* \big)
		 \cdot (dd^c V_{\Omega}^*)^i \wedge (dd^c V_{c K, \Omega}^*)^{n - i},
	\end{equation}
	where $(\cdot)^*$ denotes the upper semicontinuous regularization and the wedge products are interpreted in the Bedford-Taylor sense \cite{BedfordTaylor}.
	Note that, while the relative capacity $C(K, \Omega)$ does not depend on the ambient manifold $X$, both potentials $V_{\Omega}$ and $V_{c K, \Omega}$ depend on it crucially; in fact, they depend even on the embedding $X \subset \comp^m$, just as the Lelong class does.
	Our next result shows that, under natural assumptions, the dependence of the energy on the outer geometry disappears, as $c \to 0$.
	\begin{thm}\label{thm_rel_cap}
		For any strictly hyperconvex strictly Runge domain $\Omega \Subset X$ in an affine manifold $X \subset \comp^m$ of dimension $n$, and any non-pluripolar compact subset $K$ of $\Omega$, the following relation between the Monge-Ampère energy and the relative capacity holds
		\begin{equation}\label{eq_thm_rel_cap}
			\lim_{c \to 0} \frac{ E(V_{\Omega}) - E(V_{c K, \Omega})}{c^{n + 1}}
			=
			\frac{2 C(K, \Omega)}{(n + 1)!}.
		\end{equation}
	\end{thm}
	\begin{rem}
		It is interesting to compare Theorem \ref{thm_rel_cap} with the differentiability of the relative Monge-Ampère energy, see \cite{BermanBouckBalls}.
		The latter concerns the limit analogous to (\ref{eq_thm_rel_cap}) with $c^{n+1}$ replaced by $c$, and it does not appear to support the variation of the metric necessary to (\ref{eq_thm_rel_cap}).
	\end{rem}
	\par 
	This article is organized as follows. In Section \ref{sect_met_fd}, we establish estimates on the $\varepsilon$-entropy of finite-dimensional vector spaces. 
	In Sections \ref{sect_geom_hyp} and \ref{sect_poly}, we reduce the proof of (\ref{thm_hyperconvex_vthm}) to estimating the $\varepsilon$-entropy of certain spaces of algebraic functions, and in Section \ref{sect_balls}, we express the latter entropies in terms of the relative Monge-Amp\`ere energy. 
	Section \ref{sect_rel_cap} relates the relative Monge-Amp\`ere energy to the relative capacity and contains the proof of Theorem \ref{thm_rel_cap}. 
	Theorem \ref{thm_entr} is then proved in Section \ref{sect_est_entr}. 
	Finally, in Section \ref{sect_runge}, we discuss, as an application, a quantitative 
version of the Runge property.
	\par 
	\textbf{Notation.}
	For $f \in H^0(X)$ where $X$ is a complex manifold and $E \subset X$ is an arbitrary subset, we denote by $\| f \|_E$ the sup-norm of $f$ on $E$ (it might be equal to $+\infty$).
	\par 
	In a metric space $(M, d)$, we say that a set of points forms an $\varepsilon$-cover of $A \subset M$ if the union of $\varepsilon$-balls around the points covers $A$.
	\par 
	We denote by $d = \partial + \dbar$ the usual decomposition of the exterior derivative in terms of its $(1, 0)$ and $(0, 1)$ parts, and we set 
	\begin{equation}\label{eq_defn_dc}
		d^c := \frac{\partial - \dbar}{2 \pi \imun}.
	\end{equation}
	\par 
	On a finite dimensional vector space $V$, denote by ${\rm{vol}}$ the volume of Borel subsets of $V$ with respect to some Hermitian structure on $V$.
	Note that while each such volume individually depends on the choice of the Hermitian structure, their ratio does not.
	\par
	\textbf{Acknowledgements.}
	I would like to thank St\'ephanie Nivoche for her interest and for her comments on an earlier version of the manuscript.
	This work was supported by the CNRS, École Polytechnique, and in part by the ANR projects QCM (ANR-23-CE40-0021-01), AdAnAr (ANR-24-CE40-6184), STENTOR (ANR-24-CE40-5905-01) and CanQuantFilt (ANR-25-ERCS-0009).
	
	\section{Metric entropy in finite dimensional vector spaces}\label{sect_met_fd}
	The main goal of this section is to provide estimates on the metric entropy for finite dimensional vector spaces.
	For a vector space $V$ of dimension $v$, we fix two (not necessarily Hermitian) norms $N_i = \| \cdot \|_i$, $i = 0, 1$ on $V$.
	For a given $\varepsilon > 0$, we denote by $\mathbb{B}_{\varepsilon} \subset V$ the unit ball of the norm $\max \{ N_0, \varepsilon^{-1} \cdot N_1 \}$, and by $\mathbb{B}_{0}$ the unit ball of the norm $N_0$.
	More precisely, we prove the following result on the relation between the metric entropy, $H_{\varepsilon}(\mathbb{B}_0, N_1)$, of $\mathbb{B}_0$ in $(V, N_1)$ and the ratio of volumes of the respective unit balls.
	\begin{prop}\label{prop_entr_balls}
		For any $\varepsilon > 0$, the following estimate holds
		\begin{equation}
			\Big|
			H_{\varepsilon}(\mathbb{B}_0, N_1)
			-
			\log \Big( \frac{{\rm{vol}}(\mathbb{B}_{0})}{{\rm{vol}}(\mathbb{B}_{\varepsilon})} \Big)
			\Big|
			\leq
			100 v (\log(v) + 1).
		\end{equation}
	\end{prop}
	\par 
	Let us first establish a version of the above result for Hermitian norms $H_i = \| \cdot \|_i$, $i = 0, 1$ on $V$.
	To state it, we denote by $\mathbb{B}^H_0$ the unit ball with respect to $H_0$.
	We fix a basis $e_1, \ldots, e_v$ of $V$ which is orthonormal with respect to $H_1$ and orthogonal with respect to $H_0$.
	Define $\lambda_1, \ldots, \lambda_v \in \real$ so that $\| e_i \|_0 = \exp(\lambda_i)$.
	\begin{prop}\label{prop_entr}
		For any $\varepsilon > 0$, the following estimate holds
		\begin{equation}
			2 \sum_{i = 1}^v
			(\log (\varepsilon^{-1}) - \lambda_i)_+
			\leq
			H_{\varepsilon}(\mathbb{B}^H_0, H_1)
			\leq
			2 \sum_{i = 1}^v
			(\log (\varepsilon^{-1}) - \lambda_i)_+
			+
			2v \log(6),
		\end{equation}
		where $t_+ := \max\{0, t \}$ for any $t \in \real$.
	\end{prop}
	In the proof of Proposition \ref{prop_entr}, we use the following well-known result.
	\begin{lem}\label{lem_cov}
		Let $K, T \subset \real^d$, $d \in \nat^*$, be symmetric convex bodies. Then
		\begin{equation}\label{eq_lem_cov}
			\frac{{\rm{vol}}(K)}{{\rm{vol}}(T)} \leq N(K,T) \leq
	 		\frac{{\rm{vol}}(K + \frac{1}{2} T)}{{\rm{vol}}( \frac{1}{2} T )},
		\end{equation}
		where $N(K, T)$ is the least number of translates of $T$ covering $K$, and ${\rm{vol}}$ is the Euclidean volume.
	\end{lem}
	\begin{proof}
		The lower bound follows from the fact that any cover has total measure at least ${\rm{vol}}(K)$. 
		For the upper bound, choose a maximal family $x_1, \ldots, x_N \in K$, for which the sets $x_i +  \frac{1}{2} T$ are pairwise disjoint. 
		Since $T$ is symmetric and convex, it is immediate to see that $x_1 + T, \ldots, x_N + T$ cover $K$, and so $N(K, T) \leq N$. 
		The upper bound of (\ref{eq_lem_cov}) then follows from the fact that the disjoint sets $x_i + \frac{1}{2} T$ all lie in $K + \frac{1}{2} T$, so $N \cdot {\rm{vol}}(\frac{1}{2} T) \leq {\rm{vol}}(K + \frac{1}{2} T)$.
	\end{proof}
	\begin{proof}[Proof of Proposition \ref{prop_entr}]
		We order $\lambda_i$, $i = 1, \ldots, v$ in increasing order.
		Clearly, if $\lambda_1 \geq \log (\varepsilon^{-1})$, then $H_{\varepsilon}(\mathbb{B}^H_0, H_1) = 0$, and the inequality holds trivially.
		\par 
		Let $k = 1, \ldots, v$ be the maximal such that $\lambda_k \leq \log (\varepsilon^{-1})$.
		We then consider the orthogonal projection onto $V' := \langle e_1, \ldots, e_k \rangle$.
		We denote by $\| \cdot \|'_i$, $i = 0, 1$ the restriction of $\| \cdot \|_i$ to $V'$.
		As balls in $(V, \| \cdot \|_i)$ project onto balls on $(V', \| \cdot \|'_i)$, any cover of $\mathbb{B}^H_0$ by $\varepsilon$-balls in $\| \cdot \|_1$ will induce a cover of $\mathbb{B}'_0 := \mathbb{B}^H_0 \cap V'$ by $\varepsilon$-balls in $\| \cdot \|'_1$.
		Conversely, extending the centers by zero in the last coordinates, the $\varepsilon$-balls in $\| \cdot \|'_1$ covering $\mathbb{B}'_0$, give a $\sqrt{2} \varepsilon$-covering of $\mathbb{B}^H_0$.
		From this, it is easy to see that it suffices to establish Proposition \ref{prop_entr} when $\lambda_v \leq \log (\varepsilon^{-1})$, with $\log(3)$ in place of $\log(6)$, on which we will concentrate from now on.
		\par 
		Then the lower bound of Proposition \ref{prop_entr} follows immediately from the lower bound of Lemma \ref{lem_cov} along with the trivial fact that ${\rm{vol}}(\mathbb{B}^H_0) = \prod_{i = 1}^{v} \exp(-2 \lambda_i)$, where the volume is the Euclidean volume calculated with respect to $\| \cdot \|_1$.
		The upper bound of Proposition \ref{prop_entr} follows from the upper bound of Lemma \ref{lem_cov} and the fact that $\mathbb{B}^H_0 + \frac{1}{2} \mathbb{B}^H_1 \subset \frac{3}{2} \mathbb{B}^H_0$, where $\mathbb{B}^H_1$ is a unit ball of Hermitian norm with respect to $H_1$.
	\end{proof}
	\begin{rem}\label{rem_entr_change}
		It follows easily from the above estimate that for any Hermitian norms $H_i = \| \cdot \|_i$, $H'_i = \| \cdot \|_i$, $i = 0, 1$, on $V$, verifying $c^{-1} H_i \leq H'_i \leq c H_i$, we have 
		\begin{equation}
			\big|
				H_{\varepsilon}(\mathbb{B}_0, H_1)
				-
				H_{\varepsilon}(\mathbb{B}'_0, H'_1)
			\Big|
			\leq
			20 (\log c + 1) \cdot v.
		\end{equation}
	\end{rem}
	
	We now define the “rooftop" Hermitian norm $H_0 \vee H_1$ as follows: we choose a basis $e_1, \ldots, e_v$ of $V$ which is orthogonal for both $H_0$ and $H_1$, and define $H_0 \vee H_1$ such that $e_1, \ldots, e_v$ is orthogonal for $H_0 \vee H_1$, and normalized such that $\| e_i \|_{H_0 \vee H_1} := \max \{ \| e_i \|_{H_0}, \| e_i \|_{H_1} \}$.
	\par 
	We stress that while we obviously have $H_0 \leq H_0 \vee H_1$ and $H_1 \leq H_0 \vee H_1$, the space of Hermitian norms on $V$ is not a lattice (unless $v = 1$), i.e., there is in general no minimal Hermitian norm which majorizes both $H_0$ and $H_1$.
	\par 
	We denote by $\mathbb{B}^{\vee}_{\varepsilon}$ the unit ball with respect to $H_0 \vee \varepsilon^{-1} \cdot H_1$.
	We immediately see that Proposition \ref{prop_entr} is equivalent to the following statement
	\begin{equation}\label{eq_prop_entr_ref}
		\log \Big( \frac{{\rm{vol}}(\mathbb{B}_{0})}{{\rm{vol}}(\mathbb{B}^{\vee}_{\varepsilon})} \Big)
		\leq
		H_{\varepsilon}(\mathbb{B}^H_0, H_1)
		\leq
		\log \Big( \frac{{\rm{vol}}(\mathbb{B}_{0})}{{\rm{vol}}(\mathbb{B}^{\vee}_{\varepsilon})} \Big)
		+
		2v \log(3).
	\end{equation}
	\par 
	As already observed in \cite[(3.21)]{FinGQBig}, for Hermitian $H_0$, $H_1$, we have 
	\begin{equation}\label{eq_max_vee_comp}
		\frac{1}{\sqrt{2}} H_0 \vee H_1 \leq \max\{H_0, H_1\} \leq H_0 \vee H_1.
	\end{equation}
	\par 
	Recall now that by the John ellipsoid theorem, cf. \cite[p. 27]{PisierBook}, for any norm $N_V$ on $V$, there is a Hermitian norm $N_V^H$ on $V$, verifying 
	\begin{equation}\label{eq_john_ellips}
		N_V^H \leq N_V \leq \sqrt{2v} \cdot N_V^H.
	\end{equation}
	We are now in a position to prove Proposition \ref{prop_entr_balls}.
	\begin{proof}[Proof of Proposition \ref{prop_entr_balls}]
		It follows immediately from Remark \ref{rem_entr_change}, its analogue for volume ratios, and (\ref{eq_john_ellips}) that it suffices to prove that for any Hermitian norms $H_i = \| \cdot \|_i$, $i = 0, 1$ on $V$, in the notation of Proposition \ref{prop_entr_balls}, we have
		\begin{equation}
			\log \Big( \frac{{\rm{vol}}(\mathbb{B}_{0})}{{\rm{vol}}(\mathbb{B}_{\varepsilon})} \Big)
			\leq
			H_{\varepsilon}(\mathbb{B}_0, H_1)
			\leq
			\log \Big( \frac{{\rm{vol}}(\mathbb{B}_{0})}{{\rm{vol}}(\mathbb{B}_{\varepsilon})} \Big)
			+
			2v \log(10).
		\end{equation}
		This, however, follows immediately from (\ref{eq_prop_entr_ref}) and (\ref{eq_max_vee_comp}).
	\end{proof}

	\section{Geometry of strictly hyperconvex domains}\label{sect_geom_hyp}
	The aim of this section is to show that it suffices to establish (\ref{thm_hyperconvex_vthm}) when $X$ is an affine manifold and $\Omega$ is a strictly hyperconvex strictly Runge domain in $X$. 
	This relies on the following result.
	\begin{prop}\label{prop_runge_red}
		Let $\Omega \Subset X$ be a strictly hyperconvex domain in a Stein manifold $X$.
		Then there is a Stein domain $Y \Subset X$ such that $\Omega \Subset Y$ and $\Omega$ is a strictly Runge domain in $Y$.
	\end{prop}
	To prove Proposition \ref{prop_runge_red}, we first recall that for any subset $L$ of a complex manifold $Y$, the \textit{$Y$-holomorphic hull} $\widehat{L}_Y$ of $L$ is defined as
	\begin{equation}\label{eq_hull}
		\widehat{L}_Y
		:=
		\Big\{
			x \in Y : \text{ for any } f \in H^0(Y), |f(x)| \leq \sup_{y \in L} |f(y)|
		\Big\}.
	\end{equation}
	Let us recall the following well-known characterization of Runge domains. 
	It appears in \cite[Theorem 4.3.3]{Hormander} for domains of holomorphy $Y \subset \comp^n$ (equivalently, Stein domains in $\comp^n$; cf. \cite[Theorem 4.2.8]{Hormander}), but the proof carries over verbatim to an arbitrary Stein manifold $Y$; see \cite[proof of Theorem 5.2.10]{Hormander} or the second paragraph of \cite{DemLempShiff}.
	\begin{thm}\label{thm_runge}
		For a Stein manifold $Y$ and a Stein domain $\Omega \subset Y$, the following statements are equivalent:
		1) The domain $\Omega \subset Y$ is Runge.
		\par  
		2) For any compact subset $L \subset \Omega$, the subset $\widehat{L}_Y \cap \Omega$ is relatively compact in $\Omega$.
		\par
		3) For any compact subset $L \subset \Omega$, we have $\widehat{L}_Y = \widehat{L}_{\Omega}$.
	\end{thm}
	\begin{proof}[Proof of Proposition \ref{prop_runge_red}]
		Using the notation $\Omega'$, $\rho$ and $\Omega_r$, $r < 1$, introduced in (\ref{eq_sublevel}), we claim that $Y := \Omega_{3/4}$ satisfies the conclusion of Proposition \ref{prop_runge_red}.
		\par 
		Let us first verify that $Y$ is Stein.
		According to \cite[Theorem II]{NarasimhLevi2}, a manifold is Stein if and only if it carries a continuous strictly psh exhaustion function.
		An immediate verification shows that if $\psi$ is a continuous strictly psh exhaustion function on $X$ (which exists since $X$ is Stein), then $\chi \circ \rho + \psi$ is a continuous strictly psh exhaustion function on $Y$, where $\chi(t) := \frac{1}{3/4 - t}$, $t \in [-\infty, 3/4[$.
		\par 
		As $\rho$ is continuous, we see that $\Omega_{1/2}$ forms a relatively compact subset of $Y$.
		It only remains to establish that $\Omega_{1/2}$ is a Runge subset of $Y$.
		By Theorem \ref{thm_runge}, it suffices to show that for any compact subset $K$ of $\Omega_{1/2}$, the subset $\widehat{K}_Y$ is a compact subset of $\Omega_{1/2}$.
		Recall, however, that in Stein manifolds, holomorphic hulls coincide with psh hulls, cf.  \cite[Theorems 4.3.4, 5.1.6]{Hormander}, where the latter in the notation (\ref{eq_hull}) are defined as 
		\begin{equation}\label{eq_hull_psh}
			\widehat{K}_Y^P
			:=
			\Big\{
				x \in Y : \text{ for any } p \in \psh(Y), p(x) \leq \sup_{y \in K} p(y)
			\Big\}.
		\end{equation}
		In particular, we see that 
		\begin{equation}
			\sup_{y \in \widehat{K}_Y} \rho(y) = \sup_{y \in K} \rho(y).
		\end{equation}
		As $\sup_{y \in K} \rho(y) < 1/2$ for any compact subset $K$ of $\Omega_{1/2}$, we deduce that $\widehat{K}_Y$ is a compact subset of $\Omega_{1/2}$, finishing the proof.
	\end{proof}
	\begin{rem}\label{eq_runge_dom_emb}
		It follows from the above proof that $\Omega_r$ is Runge in $\Omega_{r'}$ for any $r < r' < 1$.
		Combining this with Proposition \ref{prop_runge_red}, we conclude that a strictly Runge domain is Runge.
	\end{rem}
	\par 
	By combining Proposition \ref{prop_runge_red} with a theorem below, we see that it suffices to establish (\ref{thm_hyperconvex_vthm}) in the setting when $X$ is an affine manifold and $\Omega$ is a strictly Runge domain in $X$.
	\begin{thm}[{Stout \cite{StoutStein}, cf. \cite[Theorem 1.6]{DemLempShiff}}]\label{thm_stout}
		Any relatively compact Runge domain in a Stein manifold is biholomorphic to a relatively compact Runge domain in some affine manifold.
	\end{thm}
	
	\section{Entropy calculation through polynomial approximations}\label{sect_poly}
	We fix a Runge domain $\Omega \Subset X$ in an affine manifold $X \subset \comp^m$ of dimension $n$, and a compact non-pluripolar subset $K$ of $\Omega$.
	Recall that $A_k[X]$ denotes the restriction to $X$ of the space of polynomials of total degree $\leq k$ on $\comp^m$.
	Let us further define
	\begin{equation}
		A_{K, k}^{\Omega}[X] := A_K^{\Omega} \cap A_k[X],
	\end{equation}
	where $A_K^{\Omega}$ was defined in (\ref{eq_ak_omeg}), and we implicitly identified $A_k[X]$ with its image in $\mathscr{C}(K)$.
	\par 
	Clearly, for any $\epsilon > 0$, $k \in \nat$, the following bound holds
	\begin{equation}\label{eq_entr_lower}
		H_{\varepsilon}(A_{K, k}^{\Omega}[X], \mathscr{C}(K))
		\leq
		H_{\varepsilon}(A_K^{\Omega}, \mathscr{C}(K)).
	\end{equation}
	The main goal of this section is to prove a version of the opposite inequality; for this, for an open subset $U \subset \Omega$, we define $A_{K, k}^{U}[X]$ analogously to $A_{K, k}^{\Omega}[X]$.
	\begin{thm}\label{thm_fin_appr}
		For any relatively compact open subset $U \Subset \Omega$ verifying $K \subset U$, there is $C > 0$ such that for any $\varepsilon \in ]0, 1]$ and any $k \in \nat^*$ verifying $k \geq C \cdot \log (\varepsilon^{-1})$, we have
		\begin{equation}
			H_{\varepsilon}(A_K^{\Omega}, \mathscr{C}(K))
			\leq
			H_{\varepsilon/4}(A_{K, k}^{U}[X], \mathscr{C}(K)).
		\end{equation}
	\end{thm}
	The proof of the above result relies on the following version of a theorem of Siciak \cite{SiciakEnvFirst} on optimal polynomial approximation of holomorphic functions, extended to the setting of affine varieties by Zeriahi \cite[Theorem 3.16 and Corollary 5.5]{ZeriahiSic}; see also \cite{LevenbergSurvAppr} for a survey of related results.
	Recall first that a complex manifold $Y$ is called \textit{holomorphically convex} if for any compact $L \subset Y$, the hull $\widehat{L}_Y$, see (\ref{eq_hull}), is compact in $Y$.
	\begin{thm}\label{thm_sic_zer}
		For any compact subset $L$ of a holomorphically convex domain $V \Subset X$, there is $\epsilon_L > 0$ such that for any $f \in H^0(V)$, we have
		\begin{equation}
			\inf_{p \in A_k[X]} \| f - p \|_{L}
			\leq 
			\exp(- \epsilon_L k) \cdot \| f \|_{V}.
		\end{equation}
	\end{thm}
	\begin{proof}[Proof of Theorem \ref{thm_fin_appr}]
		Let us apply Theorem \ref{thm_sic_zer} with $L := \overline{U}$ and $V := \Omega$ (it follows from Theorem \ref{thm_runge} that $\Omega$ is holomorphically convex, and so Theorem \ref{thm_sic_zer} applies).
		We let $C := \epsilon_L^{-1}$, where $\epsilon_L > 0$ is given by Theorem \ref{thm_sic_zer}. 
		Then for any $f \in A_K^{\Omega}$, $k \geq C \cdot \log (\varepsilon^{-1})$, there is $p \in A_k[X]$, such that
		\begin{equation}\label{eq_fp_appr}
			\| f - p \|_{L}
			\leq
			\varepsilon.
		\end{equation}
		As $\| f \|_{L} \leq 1$, we then have $\| p \|_{L} \leq 2$ for such $p$, or $p \in 2 \cdot A_{K, k}^{U}[X]$.
		Also, by the triangle inequality and (\ref{eq_fp_appr}), any set of points of $\mathscr{C}(K)$ providing an $\varepsilon$-cover of $2 \cdot A_{K, k}^{U}[X]$ also provides a $2 \varepsilon$-cover of $A_K^{\Omega}$.
		However, by rescaling, any $\varepsilon$-cover of $2 \cdot A_{K, k}^{U}[X]$ can be transformed into an $\varepsilon/2$-cover of $A_{K, k}^{U}[X]$ and inversely, finishing the proof.
	\end{proof}
	
	\section{Volumes of unit balls and entropy of polynomial approximations}\label{sect_balls}
	The main goal of this section is to show that the asymptotics of the metric entropy considered in Section \ref{sect_poly} can be expressed through the relative Monge-Ampère energy.
	\par 
	We fix a relatively compact domain $\Omega$ in an affine manifold $X \subset \comp^m$ and a compact subset $K \subset \Omega$.
	Recall that the extremal functions $V_{\Omega}$, $V_{c K, \Omega}$, $c > 0$ were defined in (\ref{eq_extr}) and (\ref{eq_rel_extr}), and the relative Monge-Ampère energy, $E(V_{\Omega}) - E(V_{c K, \Omega})$ was defined in (\ref{eq_energy}).
	\begin{prop}\label{prop_pol_appr}
		Fix $c > 0$ and consider, for any $\varepsilon > 0$, $k_{\varepsilon} \in \nat$ such that $\lim_{\varepsilon \to 0} k_{\epsilon}/\log(\varepsilon^{-1}) = c^{-1}$.
		Then the following asymptotics holds
		\begin{equation}
			\lim_{\varepsilon \to 0}
			\frac{H_{\varepsilon}(A_{K, k_{\varepsilon}}^{\Omega}[X], \mathscr{C}(K))}{(\log (\varepsilon^{-1}))^{n + 1}}
			=
			\frac{E(V_{\Omega}) - E(V_{c K, \Omega})}{c^{n + 1}}.
		\end{equation}
	\end{prop}
	\par 
	The proof of Proposition \ref{prop_pol_appr} will be based on a formula for the ratio of unit balls of sup-norms.
	To recall it, we fix a complex projective manifold $Y$, $\dim Y = n$, and a \textit{big line bundle} $L$ over $Y$, i.e., such that for 
	\begin{equation}
		n_k := \dim H^0(Y, L^{\otimes k}),
	\end{equation}
	for some $c > 0$, we have $n_k \geq c \cdot k^n$, for sufficiently large $k \in \nat$.
	\par 
	For a given upper semicontinuous metric $h^L$ on $L$, defined over a non-pluripolar compact subset $K'$ and bounded below by a continuous metric over $K'$, we denote by $P[K', h^L]$ the associated envelope, defined as 
	\begin{equation}\label{eq_v_theta}
		P[K', h^L]
		=
		\inf \Big\{ 
			h^L_0 : h^L_0 \geq h^L \text{ on $K'$, and $h^L_0$ has a psh potential}
		\Big\}.
	\end{equation}	 
	Then the metric $P[K', h^L]_*$ has a psh potential of minimal singularities, cf. \cite{DemPluripot}, \cite[Proposition 2.2]{GuedjLuZeriahEnv}.
	We then define a positive $(1, 1)$-current $c_1(L, P[K', h^L]_*) := c_1(L, h^L_0) + \ddc \phi^*$, where $h^L_0$ is an arbitrary smooth metric on $L$ and $P[K', h^L] = h^L_0 \cdot \exp(- 2 \phi)$.
	\par 
	Now, given two compact non-pluripolar subsets $K_0, K_1$ in $Y$ and upper semicontinuous metrics $h^L_0, h^L_1$, defined on $K_0$ and $K_1$ and bounded below by continuous metrics there, we define the relative Monge-Ampère energy
	\begin{multline}\label{eq_energy_proj}
		E(P[K_0, h^L_0]) - E(P[K_1, h^L_1])
		\\
		=
		\frac{1}{(n + 1)!}
		\sum_{i = 0}^{n} \int_Y \log \Big( \frac{P[K_1, h^L_1]_*}{P[K_0, h^L_0]_*} \Big)
		 \cdot c_1(L, P[K_0, h^L_0]_*)^i \wedge c_1(L, P[K_1, h^L_1]_*)^{n - i}.
	\end{multline}
	The following result was established by Berman-Boucksom \cite{BermanBouckBalls} for continuous metrics.
	We will need a version of it for upper semicontinuous metrics as above, established in \cite[Lemma 4.7]{FinKIDim} by reduction to the continuous case.
	\begin{thm}\label{thm_ber_bouck_balls}
		In the above notation, we denote by $\mathbb{B}^k_i$, $i = 0, 1$, the unit balls on $H^0(Y, L^{\otimes k})$ calculated with respect to the sup-norms associated with $(K_i, h^L_i)$.
		Then we have 
		\begin{equation}
			\lim_{k \to \infty} \frac{1}{k^{n + 1}}  \log \Big( \frac{{\rm{vol}}(\mathbb{B}^k_0)}{{\rm{vol}}(\mathbb{B}^k_1)} \Big)
			=
			E(P[K_0, h^L_0]) - E(P[K_1, h^L_1]).
		\end{equation}
	\end{thm}
	\par 
	To prove Proposition \ref{prop_pol_appr}, we need to compare various $\sup$-norms.
	For a subset $E \subset Y$ and a metric $h^L$ on $L$, defined over $E$, we denote by $\textrm{Ban}_k^{\infty}(E, h^L)$ the associated $\sup$-norm on $H^0(Y, L^{\otimes k})$, $k \in \nat^*$.
	Recall first the tautological maximum principle, cf. \cite[Proposition 1.8]{BermanBouckBalls}.
	\begin{lem}\label{lem_taut_max}
		For any subset 	$E \subset Y$ and a metric $h^L$ on $L$, defined over $E$, we have
		\begin{equation}
				\textrm{Ban}_k^{\infty}(E, h^L)
				=
				\textrm{Ban}_k^{\infty}(Y, P[E, h^L]).
		\end{equation}
	\end{lem}
	As an easy consequence of the above, we have the following result.
	\begin{lem}\label{eq_lem_comp_sup_nm}
		There is $C > 0$ such that for any $k \in \nat^*$, we have
		\begin{equation}
			\exp(- Ck) \cdot \textrm{Ban}_k^{\infty}(K_0, h^L_0) \leq \textrm{Ban}_k^{\infty}(K_1, h^L_1) \leq \exp(Ck) \cdot \textrm{Ban}_k^{\infty}(K_0, h^L_0).
		\end{equation}
	\end{lem}
	\begin{proof}
		Clearly, it suffices to show Lemma \ref{eq_lem_comp_sup_nm} in a special case when $K_1 = Y$ and $h^L_1$ is any continuous metric dominating $h^L_0$ on $K_0$.
		Then it is immediate that $\textrm{Ban}_k^{\infty}(K_0, h^L_0) \leq \textrm{Ban}_k^{\infty}(Y, h^L_1)$.
		As the potential of $P[K_0, h^L_0]_*$ is psh and of minimal singularities, cf. \cite[Proposition 2.2]{GuedjLuZeriahEnv}, there is a constant $C > 0$, such that $P[Y, h^L_1]_* \leq \exp(C) \cdot P[K_0, h^L_0]_*$.
		However, as $h^L_1$ is continuous, $P[Y, h^L_1]_*$ is one of the competitors in the definition of $P[Y, h^L_1]$.
		It then follows that $P[Y, h^L_1]_* = P[Y, h^L_1]$, and so Lemma \ref{lem_taut_max} implies that $\textrm{Ban}_k^{\infty}(Y, h^L_1) \leq \exp(Ck) \cdot \textrm{Ban}_k^{\infty}(K_0, h^L_0)$, finishing the proof.
	\end{proof}
	We also need to compare the volumes of balls on a finite dimensional vector space $V$ and its subspace $E \subset V$.
	We denote $v := \dim V$, $e := \dim E$, and assume that for some $\epsilon > 0$, we have $e/v \geq 1 - \epsilon$.
	Let $N_0$, $N_1$ be two norms on $V$ and $N'_0$, $N'_1$ be their restrictions to $E \subset V$.
	We denote by $\mathbb{B}_0$, $\mathbb{B}_1$ (resp. $\mathbb{B}'_0$, $\mathbb{B}'_1$) the unit balls on $(V, N_0)$, $(V, N_1)$ (resp. $(E, N'_0)$, $(E, N'_1)$).
	\begin{lem}\label{lem_comp_rest}
		For any $C \geq 1$, verifying $N_0 \cdot \exp(-C) \leq N_1 \leq N_0 \cdot \exp(C)$, we have
		\begin{equation}\label{eq_comp_dist_two}
			\Big| \Big( \log \Big( \frac{{\rm{vol}}(\mathbb{B}_1)}{{\rm{vol}}(\mathbb{B}_0)} \Big) - \log \Big( \frac{{\rm{vol}}(\mathbb{B}'_1)}{{\rm{vol}}(\mathbb{B}'_0)} \Big) \Big|
			\leq
			3 \epsilon \cdot C \cdot v + 40 (1 + \log v) \cdot v.
		\end{equation}
	\end{lem}
	\begin{proof}
		It is an immediate consequence of \cite[Lemmas 2.22 and 3.3]{FinGQBig}.
	\end{proof}
	\par 
	Let us now make a connection between the extremal functions and the psh envelopes.
	We denote by $\overline{X} \subset \mathbb{P}^m$ the Zariski closure of $X$, and by $L$ the restriction of the hyperplane bundle $\mathscr{O}(1)$ over $\overline{X}$.
	Even though $X$ is a manifold, $\overline{X}$ might be singular; let $\widetilde{X}$ be a resolution of singularities of its normalization.
	We denote by $\widetilde{L}$ the pull-back of $L$ to $\widetilde{X}$.
	Note that $\widetilde{L}$ is then big, as $L$ is ample and we have an embedding
	\begin{equation}
		H^0(\overline{X}, L^{\otimes k})
		\hookrightarrow
		H^0(\widetilde{X}, \widetilde{L}^{\otimes k}).
	\end{equation}
	Moreover, it is classical, cf. \cite[proof of Lemma 2.2.3]{LazarBookI}, that this embedding is “essentially" surjective, i.e. for any $\epsilon > 0$, there is $k_0 \in \nat$ such that for any $k \geq k_0$, we have
	\begin{equation}\label{eq_as_surj}
		\frac{\dim H^0(\overline{X}, L^{\otimes k})}{\dim H^0(\widetilde{X}, \widetilde{L}^{\otimes k})} 
		\geq
		1 - \epsilon.
	\end{equation}
	We fix a holomorphic section $\sigma \in H^0(\mathbb{P}^m, \mathscr{O}(1))$ whose divisor is the hyperplane at infinity.
	\par
	We denote by $h^{\widetilde{L}}_{{\rm{sing}}}$ the pull-back of the singular metric on $L$ such that for any $x \in X$, we have $\| \sigma(x) \|_{h^{\widetilde{L}}_{{\rm{sing}}}} = 1$.
	Denote by $\widetilde{K}$, $\widetilde{\Omega}$ the preimages of $K$ and $\overline{\Omega}$ in $\widetilde{X}$.
	Note that the restriction of $h^{\widetilde{L}}_{{\rm{sing}}}$ to $\Omega$ is continuous, and so $P[\widetilde{\Omega}, h^{\widetilde{L}}_{{\rm{sing}}}]$ is well defined.
	The growth conditions coming from the assumption that the functions are considered in the Lelong class then imply that the potentials of the metrics $h^{\widetilde{L}}_{{\rm{sing}}} \cdot \exp(- 2 V_{\Omega})$ and $h^{\widetilde{L}}_{{\rm{sing}}} \cdot \exp(- 2 V_{c K, \Omega})$ are bounded near the hyperplane at infinity.
	The normality of $\widetilde{X}$ implies that the upper semicontinuous regularizations of these potentials extend to $\widetilde{X}$ as psh potentials, and from the maximality, for any $c > 0$, we deduce that
	\begin{equation}\label{eq_metr_pot0}
	\begin{aligned}
		&
		h^{\widetilde{L}}_{{\rm{sing}}} \cdot \exp(- 2 V_{\Omega})
		=
		P[\widetilde{\Omega}, h^{\widetilde{L}}_{{\rm{sing}}}],
		\\
		&
		h^{\widetilde{L}}_{{\rm{sing}}} \cdot \exp(- 2 V_{c K, \Omega})
		=
		P\Big[\widetilde{\Omega}, \exp(2 c \mathds{1}_{\widetilde{K}}) h^{\widetilde{L}}_{{\rm{sing}}} \Big],
	\end{aligned}
	\end{equation}
	where $\mathds{1}_{\widetilde{K}}$ is the indicator function of $\widetilde{K} \subset \widetilde{X}$.
	This will be used in the following simple result.
	\begin{lem}\label{lem_en_cont}
		The function $c \mapsto E(V_{\Omega}) - E(V_{cK, \Omega})$ is continuous on 
$[0, +\infty[$.
	\end{lem}
	\begin{proof}
		By the cocycle property of the Monge-Ampère energy functional, cf. \cite[Corollary 3.2]{BermanBouckBalls}, for any $c, c' > 0$, we have
		\begin{multline}
			(E(V_{\Omega}) - E(V_{c K, \Omega}))
			-
			(E(V_{\Omega}) - E(V_{c' K, \Omega}))
			\\
			=
			\frac{2}{(n + 1)!}
			\sum_{i = 0}^{n} \int_X \big( V_{c' K, \Omega}^* - V_{c K, \Omega}^* \big)
			 \cdot (dd^c V_{c' K, \Omega}^*)^i \wedge (dd^c V_{c K, \Omega}^*)^{n - i}.
		\end{multline}
		Note that by (\ref{eq_metr_pot0}) and the definition of the first Chern form, we deduce
		\begin{multline}\label{eq_ddc_v_c1}
			(dd^c V_{c' K, \Omega}^*)^i \wedge (dd^c V_{c K, \Omega}^*)^{n - i}
			\\
			=
			c_1(\widetilde{L}, P\big[\widetilde{\Omega}, \exp(2 c' \mathds{1}_{\widetilde{K}}) h^{\widetilde{L}}_{{\rm{sing}}} \big])^i \wedge c_1(\widetilde{L}, P\big[\widetilde{\Omega}, \exp(2 c \mathds{1}_{\widetilde{K}}) h^{\widetilde{L}}_{{\rm{sing}}} \big])^{n - i}.
		\end{multline}
		As the integral below depends only on the singularity class of the metrics, see \cite[Theorem 1.2]{BermanBouckBalls}, and our metrics are of minimal singularities, cf. \cite[Proposition 2.2]{GuedjLuZeriahEnv}, we obtain
		\begin{equation}
			\int_{\widetilde{X}} c_1(L, P\big[\widetilde{\Omega}, \exp(2 c' \mathds{1}_{\widetilde{K}}) h^{\widetilde{L}}_{{\rm{sing}}} \big])^i \wedge c_1(L, P\big[\widetilde{\Omega}, \exp(2 c \mathds{1}_{\widetilde{K}}) h^{\widetilde{L}}_{{\rm{sing}}} \big])^{n - i}
			=
			\int_{\overline{X}} c_1(L)^n.
		\end{equation}
		It now suffices to remark that $|V_{c' K, \Omega}^* - V_{c K, \Omega}^*| \leq |c - c'|$, which finishes the proof.
	\end{proof}
	\par 
	\begin{proof}[Proof of Proposition \ref{prop_pol_appr}]
		Let $\mathcal{J}_{\overline{X}}$ denote the ideal sheaf of holomorphic functions on $\mathbb{P}^m$ vanishing along $\overline{X}$, and $\iota : \overline{X} \to \mathbb{P}^m$ the embedding.
		Consider the short exact sequence of sheaves
		\begin{equation}
			0 
			\rightarrow 
			\mathcal{J}_{\overline{X}} 
			\rightarrow
			\mathscr{O}_{\mathbb{P}^m} 
			\rightarrow
			\iota_* \mathscr{O}_{\overline{X}} 
			\rightarrow
			0.
		\end{equation}
		Tensoring the above sequence by $\mathscr{O}(1)^{\otimes k}$ and passing to the associated long exact sequence, we get
		\begin{equation}
			\cdots 
			\rightarrow 
			H^0(\mathbb{P}^m, \mathscr{O}(1)^{\otimes k})
			\rightarrow
			H^0(\overline{X}, L^{\otimes k})
			\rightarrow
			H^1(\mathbb{P}^m, \mathscr{O}(1)^{\otimes k} \otimes \mathcal{J}_{\overline{X}})
			\rightarrow
			\cdots.
		\end{equation}
		By Serre's vanishing theorem, for $k$ large enough, the last cohomology group above vanishes, and so the first map above is surjective.
		When this is combined with the classical identification of $H^0(\mathbb{P}^m, \mathscr{O}(1)^{\otimes k})$ with the space of homogeneous polynomials of degree $k$ of $m + 1$ variables, for $k$ large enough, we obtain the following well-known isomorphism
		\begin{equation}
			A_k[X] \simeq H^0(\overline{X}, L^{\otimes k}),
		\end{equation}
		which can be realized by multiplication by a $k$-th tensor power of a holomorphic section $\sigma \in H^0(\mathbb{P}^m, \mathscr{O}(1))$ whose divisor is the hyperplane at infinity.
		\par 
		For any $k \in \nat$, we denote by $N_0^k$ (resp. $N_1^k$) the sup-norm on $A_k[X]$, evaluated on $\Omega$ (resp. $K$).
		We let $\mathbb{B}_{i}^k$ be the unit ball of $N_i^k$, and let $\mathbb{B}_{\varepsilon}^k$ be the unit ball of $\max \{ N_0^k, \varepsilon^{-1} \cdot N_1^k \}$.
		By Proposition \ref{prop_entr_balls}, we conclude that 
		\begin{equation}\label{eq_entr_appr_1}
			\Big|
			H_{\varepsilon}(A_{K, k_{\varepsilon}}^{\Omega}[X], \mathscr{C}(K))
			-
			\log \Big( \frac{{\rm{vol}}(\mathbb{B}_{0}^{k_{\varepsilon}})}{{\rm{vol}}(\mathbb{B}_{\varepsilon}^{k_{\varepsilon}})} \Big)
			\Big|
			\leq
			100 n_{k_{\varepsilon}} (\log(n_{k_{\varepsilon}}) + 1).
		\end{equation}
		If we denote by $\widetilde{\mathbb{B}}_0^k$, $\widetilde{\mathbb{B}}_{\varepsilon}^k$ the unit balls defined on $H^0(\widetilde{X}, \widetilde{L}^{\otimes k})$ analogously to the above, by Lemmas \ref{eq_lem_comp_sup_nm}, \ref{lem_comp_rest} and (\ref{eq_as_surj}), we see that for any $\epsilon > 0$, there is $k_0 \in \nat$ such that for any $k \geq k_0$, we have
		\begin{equation}\label{eq_appr_big_bl}
			\Big|
			\log \Big( \frac{{\rm{vol}}(\widetilde{\mathbb{B}}_{0}^k)}{{\rm{vol}}(\widetilde{\mathbb{B}}_{\varepsilon}^k)} \Big)
			-
			\log \Big( \frac{{\rm{vol}}(\mathbb{B}_{0}^k)}{{\rm{vol}}(\mathbb{B}_{\varepsilon}^k)} \Big)
			\Big|
			\leq
			\epsilon k n_k.
		\end{equation}
		We will now show that the norms involved in (\ref{eq_appr_big_bl}) are actually supremum norms, so that the desired ratio of volumes can be computed by Theorem \ref{thm_ber_bouck_balls}.
		\par 
		Note that for $c_0 := \frac{\log(\varepsilon^{-1})}{k}$, we have the following identity
		\begin{equation}
			\max \{ N_0^k, \varepsilon^{-1} \cdot N_1^k \}
			=
			\textrm{Ban}_k^{\infty} \Big(\widetilde{\Omega}, \exp(2 c_0 \mathds{1}_{\widetilde{K}}) h^{\widetilde{L}}_{{\rm{sing}}} \Big)|_{A_k[X]}.
		\end{equation}
		Clearly, the metric $\exp(2 c_0 \mathds{1}_{\widetilde{K}}) h^{\widetilde{L}}_{{\rm{sing}}}$ is upper semicontinuous, as $c_0 > 0$.
		By this, Theorem \ref{thm_ber_bouck_balls} and (\ref{eq_ddc_v_c1}), we conclude that for any $c' < c < c''$ and $k_{\varepsilon} \in \nat$ as in Proposition \ref{prop_pol_appr}, we have
		\begin{multline}\label{eq_fin_est1}
			E(V_{\Omega}) - E(V_{c'' K, \Omega})
			\leq
			\liminf_{\varepsilon \to 0}
			\frac{1}{k_{\varepsilon}^{n + 1}}
			\log \Big( 
			\frac{{\rm{vol}}(\widetilde{\mathbb{B}}_{0}^{k_{\varepsilon}})}{{\rm{vol}}(\widetilde{\mathbb{B}}_{\varepsilon}^{k_{\varepsilon}})}
			\Big)
			\\
			\leq
			\limsup_{\varepsilon \to 0}
			\frac{1}{k_{\varepsilon}^{n + 1}}
			\log \Big( 
			\frac{{\rm{vol}}(\widetilde{\mathbb{B}}_{0}^{k_{\varepsilon}})}{{\rm{vol}}(\widetilde{\mathbb{B}}_{\varepsilon}^{k_{\varepsilon}})}
			\Big)
			\leq
			E(V_{\Omega}) - E(V_{c' K, \Omega}).
		\end{multline}
		By (\ref{eq_entr_appr_1}), (\ref{eq_appr_big_bl}) and (\ref{eq_fin_est1}), we conclude that
		\begin{multline}
			E(V_{\Omega}) - E(V_{c'' K, \Omega})
			\leq
			\liminf_{\varepsilon \to 0}
			\frac{H_{\varepsilon}(A_{K, k_{\varepsilon}}^{\Omega}[X], \mathscr{C}(K))}{k_{\varepsilon}^{n + 1}}
			\\
			\leq
			\limsup_{\varepsilon \to 0}
			\frac{H_{\varepsilon}(A_{K, k_{\varepsilon}}^{\Omega}[X], \mathscr{C}(K))}{k_{\varepsilon}^{n + 1}}
			\leq
			E(V_{\Omega}) - E(V_{c' K, \Omega}).
		\end{multline}
		This finishes the proof by Lemma \ref{lem_en_cont}.
	\end{proof}

	\section{Relative capacity and relative Monge-Ampère energy}\label{sect_rel_cap}
	The main goal of this section is to relate the (local) notion of relative capacity with the (global) notion of relative Monge-Ampère energy, that is, to establish Theorem \ref{thm_rel_cap}.
	\par 
	The proof of Theorem \ref{thm_rel_cap} is based on a combination of certain results on the relative extremal functions.
	To set up the notation, we denote by $\Omega$ a strictly hyperconvex strictly Runge domain in an affine manifold $X \subset \comp^m$, and by $K$ a non-pluripolar compact subset of $\Omega$. 
	We use the notation $\rho, \Omega', \Omega_r$, $r < 1$, introduced in (\ref{eq_sublevel}).
	Recall that the extremal functions $V_{\Omega}$, $V_{c K, \Omega}$, $c > 0$ were defined in (\ref{eq_extr}) and (\ref{eq_rel_extr}), and the relative Monge-Ampère energy, $E(V_{\Omega}) - E(V_{c K, \Omega})$, in (\ref{eq_energy}).
	\begin{prop}\label{prop_v_ck_om_coin}
		There is $c_0 > 0$ such that for any $0 < c < c_0$, outside $\Omega$, we have $V_{c K, \Omega} = V_{\Omega}$, and on $\partial \Omega$, we have $V_{c K, \Omega} = 0$.
	\end{prop}
	The proof of Proposition \ref{prop_v_ck_om_coin} will be based on the following two statements.
	\begin{lem}\label{lem_extr_fun_lem}
		For any $0 < r < 1$, the extremal function $V_{\Omega_r}$ verifies $V_{\Omega_r}^{-1}(0) = \overline{\Omega_r}$.
	\end{lem}	
	\begin{proof}
		It is well known, cf. \cite[(3.18)]{ZeriahiSic}, that $V_{\Omega_r}^{-1}(0)$ coincides with the psh hull of $\overline{\Omega_r}$ in $X$, see (\ref{eq_hull_psh}) for the definition of the latter.
		As $\Omega'$ is Runge in $X$, we conclude by Remark \ref{eq_runge_dom_emb} that the holomorphic hull, see (\ref{eq_hull}), of $\overline{\Omega_r}$ in $X$ is just $\overline{\Omega_r}$ itself.
		As psh hull is obviously contained in the holomorphic hull, this finishes the proof.
	\end{proof}
	\begin{lem}\label{lem_extr_fun_lem2}
		For any $0 < r < 1$, there is $v \in \mathscr{L}(X)$ such that $v \leq 0$ in a neighborhood of $\overline{\Omega}$, but $v \geq \delta$ for some $\delta > 0$ in a neighborhood of $\partial \Omega_r$.
	\end{lem}
	\begin{proof}
		As $\Omega$ is strictly Runge, the holomorphic hull $\widehat{\Omega}_X$ of $\Omega$ coincides with $\overline{\Omega}$ by Proposition \ref{prop_runge_red}.
		Then by \cite[Theorem 3.16]{ZeriahiSic}, for any $0 < r < 1$, there is a regular neighborhood $E$ of $\overline{\Omega}$ such that $E \subset \Omega_{r/2}$ (regularity here means that $V_E$ is continuous).
		Then $V_E = V_E^*$ is continuous and psh.
		We claim that $v := V_E$ then verifies the assumptions of Lemma \ref{lem_extr_fun_lem2}.
		Indeed, we have $v \geq V_{\Omega_{r/2}}$, and hence by Lemma \ref{lem_extr_fun_lem}, $v > 0$ in a neighborhood of $\partial \Omega_r$.
		However, as $v$ is continuous and $\partial \Omega_r$ is compact, $v \geq \delta$ for some $\delta > 0$ in a neighborhood of $\partial \Omega_r$.
		It is also immediate that $v \leq 0$ in a neighborhood of $\overline{\Omega}$, finishing the proof.
	\end{proof}
	\begin{proof}[Proof of Proposition \ref{prop_v_ck_om_coin}]
		Upon multiplication of $\rho$ by a positive constant, we may assume that $\rho \leq -1$ on $K$.
		By Lemma \ref{lem_extr_fun_lem2}, we fix $\delta > 0$ and $v \in \mathscr{L}(X)$ such that $v \leq 0$ in a neighborhood of $\overline{\Omega}$, but $v \geq \delta$ in a neighborhood of $\partial \Omega_{1/2}$.
		\par 
		For any $0 < c, \epsilon < \delta/2$, we define the following function
		\begin{equation}\label{eq_part_def_u_psh}
			u(x) 
			:=
			\begin{cases}
				c \rho(x), \quad &\text{if } x \in \Omega,
				\\
				\max \big\{ v - \epsilon,  c \rho(x) \big\}, \quad &\text{if } x \in \Omega_{1/2} \setminus \Omega,
				\\
				v - \epsilon, \quad &\text{if } x \in X \setminus \Omega_{1/2}.
			\end{cases} 
		\end{equation}
		We claim that $u \in \mathscr{L}(X)$.
		The growth condition at infinity is immediate, as $v \in \mathscr{L}(X)$.
		It only remains to establish that $u$ is psh.
		For this, cf. \cite[Corollary 2.9.15]{KlimekBook}, it suffices to show that 
		\begin{equation}
			u(x) 
			=
			\begin{cases}
				c \rho(x), \quad &\text{for $x$ in a neighborhood of $\partial \Omega$},
				\\
				v - \epsilon, \quad &\text{for $x$ in a neighborhood of $\partial \Omega_{1/2}$.}
			\end{cases} 
		\end{equation}
		But the latter follows from our choice of $v$ and from the continuity of $\rho$.
		\par 
		By the maximality of $V_{c K, \Omega}$, it then follows that $V_{c K, \Omega} \geq u$.
		As a consequence, we obtain that for any $0 < c, \epsilon < \delta/2$, the following bound holds
		\begin{equation}\label{eq_vck_low_bnd}
			V_{c K, \Omega}
			\geq
			\begin{cases}
				c \rho(x), \quad &\text{if } x \in \Omega,
				\\
				v - \epsilon, \quad &\text{otherwise.}
			\end{cases} 
		\end{equation}
		The first inequality above implies the first claim of Proposition \ref{prop_v_ck_om_coin}, as we obviously have $V_{c K, \Omega} \leq 0$ on $\overline{\Omega}$.
		By taking the supremum over all such $v$ and taking $\epsilon \to 0$ in the second inequality, we obtain the second claim of Proposition \ref{prop_v_ck_om_coin} by the obvious fact that $V_{c K, \Omega} \leq V_{\Omega}$.
	\end{proof}
	\par 
	One of the main ideas of our proof of Theorem \ref{thm_rel_cap} is to relate the extremal function $V_{c K, \Omega}$ with the \textit{relative extremal function} $u_{K, \Omega}$ of $K \subset \Omega$, defined as
	\begin{equation}\label{eq_defn_rel_extr}
		u_{K, \Omega}
		=
		\sup\big\{ u(z) : u \in \psh(\Omega),\ u \leq 0, u|_K \leq -1 \,\big\}.
	\end{equation}
	The central reason why the relative extremal function plays an important role in our work is that Bedford-Taylor in \cite{BedfordTaylor}, cf. \cite[Proposition 4.6.1]{KlimekBook}, established the following formula
	\begin{equation}\label{eq_rho_int2}
		- \int_{\Omega}  u_{K, \Omega}^* (dd^c u_{K, \Omega}^*)^{n}
		=
		C(K, \Omega).
	\end{equation}
	We will need the following auxiliary statement.
	\begin{lem}\label{eq_rel_extr_f}
		As $r \to 0$ from above, $u_{K, \Omega_r}$ converges to $u_{K, \Omega}$ pointwise in $\Omega$.
	\end{lem}
	\begin{rem}
		a) When $r \to 0$ from below, the analogous convergence also holds under milder hyperconvexity assumptions on $\Omega$, cf. \cite[Proposition 4.5.7]{KlimekBook}.
		\par 
		b) Bandtlow-Nivoche in \cite[Lemma 2.15]{BandNivoc} gave a proof of a weaker version of Lemma \ref{eq_rel_extr_f}, which fully suffices for our needs.
	\end{rem}
	\begin{proof}
		Note that in $\Omega$, we obviously have the bound $u_{K, \Omega_r} \leq u_{K, \Omega}$.
		It then suffices to show that for any $\delta, \epsilon > 0$ small enough, there is $r_0 > 0$ such that for any $0 < r < r_0$, we have $u_{K, \Omega_r} \geq u_{K, \Omega} - \epsilon$ on $\Omega_{-\delta}$.
		For this, we fix $\delta, \epsilon > 0$ such that $K$ is compact in $\Omega_{-\delta}$.
		Let $v \in \psh(\Omega)$ be such that $-1 \leq v \leq 0$ and $v|_K \leq -1$.
		We fix $A > 0$ large enough such that $A (\rho - r) < A \rho < -1 - \epsilon \leq v - \epsilon$ in a small neighborhood of $\partial \Omega_{-\delta}$.
		We take $r_0 > 0$ small enough such that we have $A (\rho - r_0) >  - \epsilon \geq v - \epsilon$ in a small inner neighborhood of $\partial \Omega$.
		Then for any $0 < r < r_0$, the following function is psh:
		\begin{equation}\label{eq_u_cnt_ur}
			u(x) 
			:=
			\begin{cases}
				v - \epsilon, \quad &\text{if } x \in \Omega_{-\delta},
				\\
				\max \big\{ A (\rho - r),  v - \epsilon \big\}, \quad &\text{if } x \in \Omega \setminus \Omega_{-\delta},
				\\
				A(\rho - r), \quad &\text{if } x \in \Omega_r \setminus \Omega.
			\end{cases} 
		\end{equation}
		As $u$ is one of the competitors for $u_{K, \Omega_r}$, by the maximality of $u_{K, \Omega_r}$, we have $u_{K, \Omega_r} \geq u$. 
		This shows that $u_{K, \Omega_r} \geq v - \epsilon$ on $\Omega_{-\delta}$. 
		However, note that in (\ref{eq_defn_rel_extr}) it suffices to consider the supremum over psh functions $u$ with range in $[-1, 0]$, as otherwise one could take $\max \{ u, -1 \}$.
		Hence, by taking the supremum over all such $v$, we obtain that $u_{K, \Omega_r} \geq u_{K, \Omega} - \epsilon$ on $\Omega_{-\delta}$, as needed.
	\end{proof}
	Let us also recall the following well-known result.
	\begin{lem}[ {\cite[Proposition 4.5.7]{KlimekBook}} ]\label{lem_cap_reg}
		As $r \to 0$ from below, $C(K, \Omega_r)$ converges to $C(K, \Omega)$.
	\end{lem}
	\begin{rem}\label{rem_c_above_cont}
		From Lemma \ref{eq_rel_extr_f}, (\ref{eq_rho_int2}) and (\ref{eq_rho_int1}), one can also easily establish that the analogous convergence holds, as $r \to 0$ from above, cf. \cite[Lemma 2.15]{BandNivoc}.
		Note that this implies (\ref{eq_cap_ext_int_eq}). 
		Indeed, if we denote by $\widehat{\Omega}_r$, $r < 1$, the sublevel sets, defined as in (\ref{eq_sublevel}) but for $\Omega := \widehat{\Omega}$, since $\widehat{\Omega}$ (and hence $\Omega$) is relatively compact in $\widehat{\Omega}_r$, for every $r > 0$, we have $\underline{C}(K, \Omega) \geq C(K, \widehat{\Omega}_r)$. 
		By letting $r \to 0$, the above discussion gives $\underline{C}(K, \Omega) \geq C(K, \widehat{\Omega})$.
		The opposite inequality follows from Theorem \ref{thm_entr}.
	\end{rem}
	Now, for any $c > 0$, we define $u_c: \Omega \to [-\infty, 0]$ so that on $\Omega$, we have
	\begin{equation}
		V_{c K, \Omega} = c \cdot u_c.
	\end{equation}
	Below, we prove the following crucial result about these functions.
	\begin{prop}
		For any $x \in \Omega$, the function $c \mapsto u_c(x)$ is decreasing.
		There is $C > 0$ such that $u_c \geq C \cdot \rho$ for any $c > 0$ small enough.
		Finally, as $c \to 0$ from above, $u_c$ converges to $u_{K, \Omega}$ pointwise on $\Omega$.
	\end{prop}
	\begin{proof}
		Note that for any $0 < \alpha < 1$, $\alpha V_{cK, \Omega}$ is one of the competitors in the definition of $V_{\alpha cK, \Omega}$, and so we have the inequality $\alpha V_{cK, \Omega} \leq V_{\alpha cK, \Omega}$, which implies the first claim.
		The second claim is an immediate consequence of (\ref{eq_vck_low_bnd}).
		\par 
		Let us now establish the last claim.
		For this, by Lemma \ref{eq_rel_extr_f}, it suffices to show that for any $r > 0$, there is $c_0 > 0$ such that for any $0 < c < c_0$, we have 
		\begin{equation}\label{eq_vck_cukr}
			V_{cK, \Omega} \geq c u_{K, \Omega_r}.
		\end{equation}
		By Lemma \ref{lem_extr_fun_lem2}, consider $\delta > 0$ and $v \in \mathscr{L}(X)$ so that $v \leq 0$ in a neighborhood of $\overline{\Omega}$, but $v \geq \delta$ in a neighborhood of $\partial \Omega_{r/2}$.
		Let $v_0 \in \psh(\Omega_r)$ be such that $-1 \leq v_0 \leq 0$ and $v_0|_K \leq -1$.
		For any $c > 0$, we define
		\begin{equation}
			u(x) 
			:=
			\begin{cases}
				c v_0, \quad &\text{if } x \in \Omega,
				\\
				\max \big\{ v - \delta/2,  c v_0 \big\}, \quad &\text{if } x \in \Omega_{r/2} \setminus \Omega,
				\\
				v - \delta/2, \quad &\text{if } x \in X \setminus \Omega_{r/2}.
			\end{cases} 
		\end{equation}
		We claim that for $c > 0$ small enough (independent of $v_0$), $u \in \mathscr{L}(X)$. 
		As after (\ref{eq_part_def_u_psh}), it suffices to show that 
		\begin{equation}
			u(x) 
			=
			\begin{cases}
				c v_0, \quad &\text{for $x$ in a neighborhood of $\partial \Omega$},
				\\
				v - \delta/2, \quad &\text{for $x$ in a neighborhood of $\partial \Omega_{r/2}$.}
			\end{cases} 
		\end{equation}
		Note that the first condition is always verified as long as $c < \delta/2$.
		The second condition is also verified, since $v_0 \leq 0$ and $v \geq \delta$ in a neighborhood of $\partial \Omega_{r/2}$.
		As $u$ is one of the competitors for $V_{cK, \Omega}$, we deduce by the maximality that $V_{cK, \Omega} \geq c v_0$ on $\Omega$.
		By an argument similar to the one after (\ref{eq_u_cnt_ur}), we deduce (\ref{eq_vck_cukr}).
	\end{proof}
	
	\begin{proof}[Proof of Theorem \ref{thm_rel_cap}]
		By Proposition \ref{prop_v_ck_om_coin}, we have $V_{\Omega} = V_{cK, \Omega}$ outside $\Omega$ for $c > 0$ small enough. Moreover, by a result of Bedford-Taylor \cite{BedfordTaylor}, the sets where $V_{\Omega}^* \neq V_{\Omega}$ and where $V_{cK, \Omega}^* \neq V_{cK, \Omega}$ are pluripolar, and pluripolar sets are not charged by Monge-Amp\`ere-type measures. Hence
		\begin{equation}
			E(V_{\Omega}) - E(V_{c K, \Omega})
			=
			\frac{2}{(n + 1)!}
			\sum_{i = 0}^{n} \int_{\Omega} \big( V_{\Omega}^* - V_{c K, \Omega}^* \big)
			 \cdot (dd^c V_{\Omega}^*)^i \wedge (dd^c V_{c K, \Omega}^*)^{n - i}.
		\end{equation}
		However, as by Lemma \ref{lem_extr_fun_lem}, $V_{\Omega} = 0$ on $\Omega$ (and so $V_{\Omega}^* = 0$ there, as $\Omega$ is open), we deduce that
		\begin{equation}
			E(V_{\Omega}) - E(V_{c K, \Omega})
			=
			\frac{-2}{(n + 1)!}
			\int_{\Omega} V_{c K, \Omega}^* (dd^c V_{c K, \Omega}^*)^{n}.
		\end{equation}
		But it follows from the balayage argument, cf. \cite[Corollary 9.2]{BedfordTaylor} or \cite[Lemma 2.3]{GuedjLuZeriahEnv}, that $dd^c V_{c K, \Omega}^*$ charges only the subset of $K$ where $V_{c K, \Omega}^* = - c$ and the subset of $\overline{\Omega}$ where $V_{c K, \Omega}^* = 0$.
		For a bump function $\nu: \Omega \to [0, 1]$ with compact support and $\nu|_K = 1$, we hence have
		\begin{equation}\label{eq_rho_int1}
			\int_{\Omega} V_{c K, \Omega}^* (dd^c V_{c K, \Omega}^*)^{n}
			=
			c^{n + 1}
			\int_{\Omega} \nu \cdot u_c^* (dd^c u_c^*)^{n}.
		\end{equation}
		However, by the monotonicity of $u_c$, the usual continuity properties of the Monge-Ampère operator, cf. \cite[Proposition 5.2]{BedfordTaylor}, and Lemma \ref{eq_rel_extr_f}, we conclude that the sequence of measures $u_c^* (dd^c u_c^*)^{n}$ converges weakly towards $u_{K, \Omega}^* (dd^c u_{K, \Omega}^*)^{n}$, as $c \to 0$.
		Note that for similar reasons as in (\ref{eq_rho_int1}), we have $\int_{\Omega} \nu \cdot u_{K, \Omega}^* (dd^c u_{K, \Omega}^*)^{n} = \int_{\Omega} u_{K, \Omega}^* (dd^c u_{K, \Omega}^*)^{n}$.
		A combination of the above statements and (\ref{eq_rho_int2}) finishes the proof.
	\end{proof}
	
	\section{Estimates on the entropy of the space of holomorphic functions}\label{sect_est_entr}
	The goal of this section is to prove the main result of this article, Theorem \ref{thm_entr}. 
	As explained in the Introduction, we will first establish (\ref{thm_hyperconvex_vthm}) and then deduce Theorem \ref{thm_entr} from it.
	\begin{proof}[Proof of (\ref{thm_hyperconvex_vthm})]
		Note first that by the results of Section \ref{sect_geom_hyp}, it suffices to establish (\ref{thm_hyperconvex_vthm}) in the setting when $X$ is an affine manifold and $\Omega$ is a strictly Runge domain in $X$.
		\par 
		We fix $c > 0$.
		By applying (\ref{eq_entr_lower}) for $k := \lfloor c^{-1} \log(\varepsilon^{-1}) \rfloor$ and using Proposition \ref{prop_pol_appr}, we obtain 
		\begin{equation}\label{eq_entr_low_bnd}
			\liminf_{\varepsilon \to 0}
			\frac{H_{\varepsilon}(A_K^{\Omega}, \mathscr{C}(K))}{(\log (\varepsilon^{-1}))^{n + 1}}
			\geq
			\frac{E(V_{\Omega}) - E(V_{c K, \Omega})}{c^{n + 1}}.
		\end{equation}
		Letting $c \to 0$ in the above inequality and applying Theorem \ref{thm_rel_cap}, we obtain
		\begin{equation}
			\liminf_{\varepsilon \to 0}
			\frac{H_{\varepsilon}(A_K^{\Omega}, \mathscr{C}(K))}{(\log (\varepsilon^{-1}))^{n + 1}}
			\geq
			\frac{2 C(K, \Omega)}{(n + 1)!}.
		\end{equation}
		On the other hand, by Theorem \ref{thm_fin_appr} and Proposition \ref{prop_pol_appr}, we conclude that for any $r < 0$ small enough such that $K \subset \Omega_r$, there is $c_r > 0$ such that for any $0 < c < c_r$, we have
		\begin{equation}
			\limsup_{\varepsilon \to 0}
			\frac{H_{\varepsilon}(A_K^{\Omega}, \mathscr{C}(K))}{(\log (\varepsilon^{-1}))^{n + 1}}
			\leq
			\frac{E(V_{\Omega_r}) - E(V_{c K, \Omega_r})}{c^{n + 1}}.
		\end{equation}
		Letting $c \to 0$ in the above inequality and applying Theorem \ref{thm_rel_cap}, we obtain
		\begin{equation}
			\limsup_{\varepsilon \to 0}
			\frac{H_{\varepsilon}(A_K^{\Omega}, \mathscr{C}(K))}{(\log (\varepsilon^{-1}))^{n + 1}}
			\leq
			\frac{2 C(K, \Omega_r)}{(n + 1)!}.
		\end{equation}
		Letting $r \to 0$ from below and applying Lemma \ref{lem_cap_reg}, we conclude.
	\end{proof}
	\begin{proof}[Proof of Theorem \ref{thm_entr}]
		Note first that since $K$ is a compact subset of $\Omega$, and $\Omega$ is a subset of $\widehat{\Omega}$, $K$ may be regarded as a compact subset of $\widehat{\Omega}$.
		Note also that if $f \in H^0(\Omega)$ and $\widehat{f} \in H^0(\widehat{\Omega})$ are such that $\widehat{f}|_{\Omega} = f$, then $\sup_{x \in \widehat{\Omega}} |\widehat{f}(x)| = \sup_{y \in \Omega} |f(y)|$ (assume for the sake of contradiction that there is a point $x \in \widehat{\Omega}$ such that $|\widehat{f}(x)| > \sup_{y \in \Omega} |f(y)|$; then $1/(f - \widehat{f}(x))$ is holomorphic on $\Omega$, but it does not extend holomorphically to $\widehat{\Omega}$, contradicting the definition of $\widehat{\Omega}$).
		Hence, viewed as subspaces of $\mathscr{C}(K)$, $A_K^{\Omega}$ and $A_K^{\widehat{\Omega}}$ are isometric.
		In particular, for any $\varepsilon > 0$, we have
		\begin{equation}\label{eq_sol_main_1}
			H_{\varepsilon}(A_K^{\Omega}, \mathscr{C}(K))
			=
			H_{\varepsilon}(A_K^{\widehat{\Omega}}, \mathscr{C}(K)).
		\end{equation}
		But $\widehat{\Omega}$ is Stein, and so it admits an exhaustion $\widehat{\Omega}_i$, $i \in \nat$, by relatively compact strictly pseudoconvex domains, see before (\ref{eq_rel_cap_gen_dd}).
		After discarding finitely many initial terms of the exhaustion, we may assume that $K \subset \widehat{\Omega}_i$ for all $i \in \nat$.
		It is clear that for any $i \in \nat$, we have
		\begin{equation}\label{eq_sol_main_2}
			H_{\varepsilon}(A_K^{\widehat{\Omega}_i}, \mathscr{C}(K))
			\geq 
			H_{\varepsilon}(A_K^{\widehat{\Omega}}, \mathscr{C}(K)).
		\end{equation}
		Note that every strictly pseudoconvex domain is strictly hyperconvex.
		Combining this with (\ref{thm_hyperconvex_vthm}), (\ref{eq_sol_main_1}) and (\ref{eq_sol_main_2}), we conclude that
		\begin{equation}
			\frac{2 C(K, \widehat{\Omega}_i)}{(n + 1)!}
			\geq
			\limsup_{\varepsilon \to 0}
			\frac{H_{\varepsilon}(A_K^{\Omega}, \mathscr{C}(K))}{(\log(\varepsilon^{-1}))^{n + 1}}.
		\end{equation}
		The upper bound of Theorem \ref{thm_entr} then follows by letting $i \to \infty$ and using the definition (\ref{eq_rel_cap_gen_dd}).
		\par 
		To establish the lower bound, we observe that for any strictly hyperconvex Stein domain $U$ containing $\Omega$, we have $H_{\varepsilon}(A_K^{\Omega}, \mathscr{C}(K))
			\geq
			H_{\varepsilon}(A_K^{U}, \mathscr{C}(K))$.
		By this and (\ref{thm_hyperconvex_vthm}), we deduce
		\begin{equation}
			\liminf_{\varepsilon \to 0}
			\frac{H_{\varepsilon}(A_K^{\Omega}, \mathscr{C}(K))}{(\log(\varepsilon^{-1}))^{n + 1}}
			\geq
			\frac{2 C(K, U)}{(n + 1)!}.
		\end{equation}
		By taking the supremum over all such $U$, we deduce the result, up to one difference: the supremum in (\ref{eq_und_c}) is now taken over strictly hyperconvex Stein domains $\Omega'$ containing $\Omega$. 
		However, by Lemma \ref{lem_cap_reg}, this modification does not alter the value of $\underline{C}(K, \Omega)$.
	\end{proof}
	
	\section{On a quantitative version of the Runge property}\label{sect_runge}
	The main goal of this section is to describe one application of the main results of this paper, which concerns a quantitative version of the Runge property.
	To explain this, we fix a domain $\Omega$ in a Stein manifold $X$ of dimension $n$ and a compact non-pluripolar subset $K$ of $\Omega$.
	Along with the functional space $A_K^{\Omega}$ defined in (\ref{eq_ak_omeg}), let us consider
	\begin{equation}\label{eq_ak_omeg_x}
		A_K^{\Omega}[X] := 
		A_K^{\Omega} \cap H^0(X),
	\end{equation}
	where we identified the space $H^0(X)$ of holomorphic functions on $X$ with their restriction to $K$.
	Of course, if $\Omega$ is strictly Runge in $X$, $A_K^{\Omega}[X]$ forms a dense subset of $A_K^{\Omega}$, and so for every $\varepsilon > 0$,
	\begin{equation}
		H_{\varepsilon}(A_K^{\Omega}, \mathscr{C}(K)) = H_{\varepsilon}(A_K^{\Omega}[X], \mathscr{C}(K)).
	\end{equation}
	When $\Omega$ is not strictly Runge, this identity may fail; the discrepancy between the $\varepsilon$-entropies of the above spaces can thus be viewed as a measure of the extent to which $\Omega$ fails to be strictly Runge.
	\par 
	The asymptotics of $H_{\varepsilon}(A_K^{\Omega}, \mathscr{C}(K))$ as $\varepsilon \to 0$ was determined in Theorem \ref{thm_entr}; the main goal of this section is the analogous calculation for $H_{\varepsilon}(A_K^{\Omega}[X], \mathscr{C}(K))$. 
	More precisely, we will show that Theorem \ref{thm_entr} continues to hold in this new setting with the envelope of holomorphy replaced by the holomorphic hull. 
	To state this result, we denote by $\Omega_X$ the connected component of the interior of the holomorphic hull $\widehat{\Omega}_X$ of $\Omega$ that contains $\Omega$.
	Recall the following statement.
	\begin{lem}\label{lem_interior}
		The interior of a holomorphically closed subset $L$ (i.e. $L = \widehat{L}_X$) in a Stein manifold $X$ is Stein.
	\end{lem}
	\begin{proof}
		Note that any holomorphically closed subset $L$ in a Stein manifold $X$ can be written as an intersection of Stein domains.
		Indeed, we have
		\begin{equation}
			L = \cap_{f \in H^0(X)} \cap_{\epsilon > 0} \Big\{ x \in X : |f(x)| < \sup_{y \in L} |f(y)| + \epsilon \Big\}.
		\end{equation}
		It is hence sufficient to establish that in a Stein manifold $X$, the interior of an intersection of Stein domains is Stein.
		When $X = \comp^n$, this is well known, see \cite[Theorem 2.5.5 and Corollary 2.5.7]{Hormander}. 
		The general case can be reduced to that of $\comp^n$ by a theorem of Docquier-Grauert \cite{DocquierGrauert} stating that every locally Stein open subset of a Stein manifold is Stein, cf. \cite[Theorem 4.3]{SiuPseudoconvBull} for details.
	\end{proof}
	By Lemma \ref{lem_interior}, the capacity $C(K, \Omega_X)$ is well defined by (\ref{eq_rel_cap_gen_dd}).
	We can now state the main result of this section.
	\begin{thm}\label{thm_entr_rung}
		Let $\Omega$ be a relatively compact domain in $X$. Then we have
		\begin{multline}\label{eq_thm_entr_rung}
			\frac{2 C(K, \Omega_X)}{(n + 1)!}
			\geq
			\limsup_{\varepsilon \to 0}
			\frac{H_{\varepsilon}(A_K^{\Omega}[X], \mathscr{C}(K))}{(\log(\varepsilon^{-1}))^{n + 1}}
			\\
			\geq
			\liminf_{\varepsilon \to 0}
			\frac{H_{\varepsilon}(A_K^{\Omega}[X], \mathscr{C}(K))}{(\log(\varepsilon^{-1}))^{n + 1}}
			\geq
			\frac{2 \underline{C}[X](K, \Omega_X)}{(n + 1)!},
		\end{multline}
		where $\underline{C}[X](K, \Omega_X)$ is defined analogously to (\ref{eq_und_c}) but solely for domains $\Omega'$ inside of $X$.
	\end{thm}
	\begin{rem}
		If $\Omega_X$ is strictly hyperconvex in $X$, both sides of (\ref{eq_thm_entr_rung}) coincide by the same argument as in Remark \ref{rem_c_above_cont}. 
		By this and Hartogs' theorem, when $n \geq 2$, Theorem \ref{thm_entr_rung} allows one to calculate the asymptotics of the metric entropy of the space $A_K^{\Omega}[X]$ when $\Omega := \Omega_0 \setminus L$, where $\Omega_0$ is a strictly hyperconvex strictly Runge domain in $X$, and $L$ is any compact subset of $\Omega_0$. 
	\end{rem}
	\begin{proof}
		By the definition of the holomorphic hull, we have an isometry between $A_K^{\Omega}[X]$ and $A_K^{\Omega_X}[X]$.
		We also have an obvious upper bound
		\begin{equation}
			H_{\varepsilon}(A_K^{\Omega_X}[X], \mathscr{C}(K))
			\leq
			H_{\varepsilon}(A_K^{\Omega_X}, \mathscr{C}(K)),
		\end{equation}
		implying the upper bound of Theorem \ref{thm_entr_rung} by Theorem \ref{thm_entr}, Lemma \ref{lem_interior} and the fact that $\Omega_X$ coincides with its envelope of holomorphy, as it is Stein, cf. \cite[Theorem 5.4.2]{Hormander}.
		\par 
		Let us now establish that for an arbitrary domain $V$ in $X$, containing $\Omega_X$ as a relatively compact subset, we have
		\begin{equation}\label{eq_n_low_bnd}
			N_{\varepsilon}(A_K^{\Omega_X}[X], \mathscr{C}(K))
			\geq
			N_{\varepsilon}(A_K^{V}, \mathscr{C}(K)).
		\end{equation}
		Note that as $\Omega$ is relatively compact, $\widehat{\Omega}_X$ is relatively compact as well.
		Then by Oka-Weil theorem, cf. \cite[Theorem 18]{FornaessMergel} or \cite[Corollary 5.2.9]{Hormander}, for any $f \in A_K^{V}$, there is a sequence $f_i \in H^0(X)$, $i \in \nat$ such that 
		\begin{equation}\label{eq_fi_f_app}
			\| f_i - f \|_{\Omega_X} \to 0.
		\end{equation}
		Assume $p_1, \ldots, p_N \in \mathscr{C}(K)$ provide an $\varepsilon$-cover of $A_K^{\Omega_X}[X]$.
		Let us show that $f$ lies in an $\varepsilon$-ball around one of $p_1, \ldots, p_N$, thereby implying (\ref{eq_n_low_bnd}).
		As $\| f \|_{\Omega_X} \leq 1$, by (\ref{eq_fi_f_app}), we can find $a_i > 0$, $i \in \nat$ such that $a_i \to 1$, as $i \to \infty$, and such that for $i \in \nat$ sufficiently large, we have $\| a_i f_i \|_{\Omega_X} < 1$.
		But then there is $j = 1, \ldots, N$ such that for infinitely many $i \in \nat$, we have
		\begin{equation}
			\| p_j - a_i f_i \|_K \leq \varepsilon.
		\end{equation}
		By passing to the limit $i \to \infty$, we deduce that $\| p_j - f \|_K \leq \varepsilon$, finishing the proof.
		\par 
		A combination of (\ref{thm_hyperconvex_vthm}) and (\ref{eq_n_low_bnd}) implies that if $V$ is a strictly hyperconvex domain, then
		\begin{equation}
			\liminf_{\varepsilon \to 0}
			\frac{H_{\varepsilon}(A_K^{\Omega}[X], \mathscr{C}(K))}{(\log(\varepsilon^{-1}))^{n + 1}}
			\geq
			\frac{2 C(K, V)}{(n + 1)!}.
		\end{equation}
		By taking the supremum over all such $V$, we conclude.
	\end{proof}

\bibliography{bibliography}

		\bibliographystyle{abbrv}

\Addresses

\end{document}